\documentclass[a4paper]{amsart}

\usepackage[T1]{fontenc}
\usepackage[latin1]{inputenc}
\usepackage[ngerman,english]{babel}
\usepackage{graphicx}
\usepackage{amsmath}
\usepackage{amssymb}
\usepackage{amsfonts, dsfont}
\usepackage{mdwlist}
\usepackage{paralist}
\usepackage{hypbmsec}
\usepackage[colorlinks, linkcolor=black, citecolor=black, bookmarksopen=true, bookmarksopenlevel=4, final]{hyperref}  
\usepackage{aliascnt}
\usepackage{amsthm}
\usepackage{cleveref}
\usepackage[all]{xy}
\usepackage{mathtools}

\DeclareMathOperator{\high}{high}
\DeclareMathOperator{\medium}{medium}

\DeclareMathOperator{\homog}{homog}

\DeclareMathOperator{\sing}{sing}

\DeclareMathOperator{\Proj}{Proj }
\DeclareMathOperator{\Z}{\mathbb Z}
\DeclareMathOperator{\F}{\mathbb F}
\DeclareMathOperator{\A}{\mathbb A}

\DeclareMathOperator{\mO}{\mathcal O}
\DeclareMathOperator{\mP}{\mathbb P}
\DeclareMathOperator{\mfm}{\mathfrak m}
\DeclareMathOperator{\Prob}{\mbox{Prob}}

\newtheorem{thm}{Theorem}[section]

\newaliascnt{defin}{thm}

\aliascntresetthe{defin}

\newaliascnt{lemma}{thm}
\newtheorem{lemma}[lemma]{Lemma}
\aliascntresetthe{lemma}

\newaliascnt{cor}{thm}

\aliascntresetthe{cor}

\newaliascnt{nota}{thm}

\aliascntresetthe{nota}

\newaliascnt{conject}{thm}

\aliascntresetthe{conject}

\newaliascnt{prop}{thm}

\aliascntresetthe{prop}

\newaliascnt{fact}{thm}

\aliascntresetthe{fact}

\newaliascnt{result}{thm}

\aliascntresetthe{result}

\newaliascnt{lemmadefin}{thm}

\aliascntresetthe{lemmadefin}

\theoremstyle{remark}
\newaliascnt{rem}{thm}
\newtheorem{rem}[rem]{Remark}
\aliascntresetthe{rem}

\theoremstyle{remark}
\newaliascnt{notation}{thm}

\aliascntresetthe{notation}

\theoremstyle{remark}
\newaliascnt{example}{thm}

\aliascntresetthe{example}

\crefname{equation}{Equation}{Equations}
\creflabelformat{equation}{#2(#1)#3}

\crefformat{thm}{#2Theorem~#1#3} 
\crefformat{defin}{#2Definition~#1#3} 
\crefformat{lemma}{#2Lemma~#1#3}
\crefformat{example}{#2Example~#1#3}
\crefformat{rem}{#2Remark~#1#3}
\crefformat{cor}{#2Corollary~#1#3} 
\crefformat{nota}{#2Notation~#1#3}
\crefformat{claim}{#2Claim~#1#3}
\crefformat{section}{#2Section~#1#3}
\crefformat{conject}{#2Conjecture~#1#3}
\crefformat{prop}{#2Proposition~#1#3}
\crefformat{lemmadefin}{#2Lemma/Definition~#1#3}

\title{The singular locus of hypersurface sections containing a closed subscheme over finite fields}
\author{Franziska Wutz}

\begin{document}
\pagestyle{plain}

\begin{abstract}
We prove that there exist hypersurfaces that contain a given closed subscheme $Z$ of the projective space over a finite field and intersect a given smooth scheme $X$ off of $Z$ smoothly, if the intersection $V=Z\cap X$ is smooth. Furthermore, we can give a bound on the dimension of the singular locus of the hypersurface section and prescribe finitely many local conditions on the hypersurface. This is an analogue of a Bertini theorem of Bloch over finite fields and is proved using Poonen's closed point sieve. We also show a similar theorem for the case where $V$ is not smooth.
\end{abstract}
\maketitle

\section{Introduction}

The classical Bertini theorem over infinite fields guarantee the existence of a smooth hypersurface section for a smooth subscheme $X$ of the projective space. For a given closed subvariety $Z\subseteq X$, Bloch showed that the hypersurface can be assumed to contain $Z$, if $V=Z\cap X$ is smooth and $2 p> \dim X$, where $1<p$ is the codimension of $Z$ in $X$ (\cite{bloch}). If the condition on the dimension is not fulfilled, this does not hold anymore. But Bloch showed that there still exists a hypersurface section $Y$ of $X$ such that the singular locus of $Y$ is smooth, contained in $Z$ and of dimension $n-2p$. 
Over finite fields, Poonen proved an analogue for the case where $V$ is smooth and $2\dim V<\dim X$ (\cite{poonensubscheme}); here $Z$ is a closed subscheme of the projective space. In this paper, we generalize this and show an analogue over finite fields of Bloch's result for $2\dim V\geq \dim X$, where we also add the possibility to impose finitely many local conditions on the hypersurface (\cref{bertinimove}).

Furthermore, we show that if the intersection $V$ is not smooth and $2\dim V=\dim X$, there exists a hypersurface $H$ containing $Z$ and intersecting $X$ off of $Z$ smoothly, such that the singular locus of $H\cap X$ consists at most of finitely many points. We use Poonen's closed point sieve (cf. \cite{poonen}) to prove this, and also prescribe finitely many local conditions on the hypersurface (\cref{bertinivnotsmooth}).\vspace{\baselineskip}
 
We use the following notation: let $\F_q$ be a finite field of $q=p^a$ elements. Let $S=\mathbb F_q\left[x_0,\ldots,x_n\right]$ and let $S_d\subseteq S$ the $\mathbb F_q$-subspace of homogeneous polynomials of degree $d$. Let $S_{\homog}=\bigcup_{d\geq 0}S_d$ and let $S_d'$ be the set of all polynomials in $\F_q[x_0, \ldots x_n]$ of degree $\leq d$.

For a scheme $X$ of finite type over $\F_q$, we define the zeta function of $X$ as
$$\zeta_X(s):=\prod\limits_{P\in X \text{ closed}} (1-q^{-s\deg P})^{-1}.$$
This product converges for $\mbox{Re}(s)>\dim X$. 

Let $Z$ be a fixed closed subscheme of $\mP^n=\mP^n_{\F_q}$. For $d\in \Z_{\geq 0}$ let $I_d$ be the $\F_q$-subspace of polynomials $f\in S_d$ vanishing on $Z$, and $I_{\homog}=\bigcup_{d\geq 0}I_d$. For a polynomial $f\in I_d$ let $H_f=\Proj(S/(f))$ be the hypersurface defined by $f$.  

As in Poonen's paper (\cite{poonensubscheme}), we want to measure the density of a set of polynomials within the space of polynomials vanishing on $Z$, and define the density relative to the closed subscheme $Z$ of a subset $\mathcal P\subseteq I_{\homog}$ to be
$$\mu_Z(\mathcal P):=\lim\limits_{d\rightarrow \infty}\frac{\#(\mathcal P\cap I_d)}{\#I_d},$$
if the limit exists. \vspace{\baselineskip}

{\bf Acknowledgments.} The author would like to thank Uwe Jannsen heartily for his support and help. Furthermore, the author gratefully acknowledges support from the Deutsche Forschungsgemeinschaft through the SFB 1085 \textit{Higher Invariants}.

\section{Results for V smooth}
\label{sec:vsmooth}
Following \cite{poonen}, we want to prescribe finitely many local conditions on the hypersurface. For this, let $Y$ be a finite subscheme of $\mP^n$. For a polynomial $f\in I_d$ we define $f\big|_Y\in H^0(Y,\mathcal I_Z \cdot \mO_Y)$ as follows: on each connected component $Y_i$ of $Y$ let $f\big|_Y$ be equal to the restriction of $x_j^{-d}f$ to $Y_i$, where $j=j(i)$ is the smallest $j\in \left\{0,1,\ldots,n\right\}$ such that the coordinate $x_j$ is invertible on $Y_i$.

\begin{thm}
\label{bertinimove}
Let $X$ be a quasi-projective subscheme of $\mP^n$ and $Y$ a finite subscheme of $\mP^n$ such that $U=X-(X\cap Y)$ is smooth of dimension $m\geq 0$ over $\F_q$. Let $Z$ be a closed subscheme of $\mP^n$ such that $Z\cap Y=\emptyset$. Let $T\subseteq H^0(Y,\mathcal I_Z\cdot \mO_Y)$. Assume that the intersection $V:=Z\cap U$ is smooth of dimension $l$ such that $2l\leq m+k$ where $k\in \Z_{\geq 0}$. Define
\begin{align*}
\mathcal P =\{f\in I_{\homog}: \;&H_f\cap U \mbox{ is smooth of dimension $m-1$ at all points } P\in U-V,\\
&f\big|_Y\in T \mbox{ and } \dim(H_f\cap U)_{\sing}\leq k\}.
\end{align*}
Then $\mu_Z(\mathcal P)=\dfrac{\# T}{\# H^0(Y,\mathcal I_Z\cdot\mO_Y)}\zeta_{U-V}(m+1)^{-1}$.
\end{thm}

In particular, there exists a hypersurface $H$ containing $Z$, defined by a polynomial $f$ such that $f\big|_Y\in T$ and such that the singular locus of the hypersurface section is contained in $Z$ and at most of dimension $k$. The existence of such a hypersurface has been shown for infinite fields as well (\cite{bloch}, Proposition 1.2).

\begin{rem}
\label{rem:notation}
The proof is a closed point sieve as introduced in \cite{poonen} and is organized as follows:
First we look at closed points of low degree, defining the relevant set of polynomials to be 
\begin{align*}
\mathcal P_r=\{f\in I_{\homog}:\; & H_f\cap U \mbox{ is smooth of dimension } m-1 \\
&\mbox{at all points } P\in (U-V)_{<r} \mbox{ and } f\big|_Y\in T\}
\end{align*}
and calculate the density of $\mathcal P_r$ in \cref{low}. As in \cite{poonen}, the argument does not work if we let $r$ tend to infinity before we measure the density. Hence we also need to consider closed points of medium and high degree. We fix $c$ such that $S_1I_d=I_{d+1}$ for all $d\geq c$ (cf. \cite{poonensubscheme}).
For closed points of medium degree, let
\begin{align*}
\mathcal Q_r^{\medium}=\bigcup\limits_{d\geq 0}\{f\in I_d:\;&\mbox{there exists a point } P\in U-V \mbox{ with } r\leq \deg P\leq \frac{d-c}{m+1} \\
&\hspace{-1mm}\mbox{such that } H_f\cap U \mbox{ is not smooth of dimension $m-1$ at } P\}.
\end{align*}
For closed points of high degree, we differentiate between points on and off of $V$, and define 
\begin{align*}
\mathcal Q^{\high}_{U-V}=\bigcup\limits_{d\geq 0}\{f\in I_d:\; &\mbox{there exists a point } P\in (U-V)_{>\frac{d-c}{m+1}} \mbox{ such that} \\
&H_f\cap U \mbox{ is not smooth of dimension $m-1$ at $P$}\},
\end{align*}
and
$$\mathcal Q_V=\{f\in I_{\homog}:\;\dim((H_f\cap U)_{\sing}\cap V)\geq k+1\}.$$
We will show in \ref{medium}, \ref{highoffv} and \ref{highonv}, that the error in the approximation that we make by considering $\mathcal P_r$ instead of $\mathcal P$ tends to zero for $r\rightarrow \infty$, i.e. the sets $Q_r^{\medium}$, $Q^{\high}_{U-V}$ and $Q_V$ are of density zero for $r\rightarrow \infty$.
\end{rem}

\subsection{Points of low degree}

\begin{lemma}
\label{lowdegree}
Suppose $\mfm\subseteq \mO_U$ is the ideal sheaf of a closed point $P\in U$. Let $C\subseteq U$ be the closed subscheme whose ideal sheaf is $\mfm^2\subseteq \mO_U$. Then for any $d\in \Z_{\geq 0}$ we have 
$$H^0(X, \mathcal I_Z\cdot\mO_C(d))= \left\{\begin{array}{cl} q^{(m-l)\deg P}, &\mbox{ if } P\in V, \\ q^{(m+1)\deg P}, &\mbox{ else}.\end{array} \right.$$
\end{lemma}
\begin{proof}
This is Lemma 2.2 of \cite{poonensubscheme}; the condition on the dimension is not needed here.
\end{proof}

\begin{lemma}[Singularities of low degree]
\label{low}
For $\mathcal P_r$ defined as in \cref{rem:notation}, 
$$\mu_Z(\mathcal P_r)=\dfrac{\# T}{\# H^0(Y,\mathcal I_Z\cdot\mO_Y)}\prod\limits_{P\in (U-V)_{<r}}(1-q^{-(m+1)\deg P})^{-1}.$$
\end{lemma}

\begin{proof}
The proof is parallel to the one of Lemma 2.4 in \cite{wutz}, just use \cref{lowdegree} above and $(U-V)_{<r}=U_{<r}-V$ instead of $U_{<r}$.
\end{proof}

\begin{rem}
\label{explain:low}
Note that the condition on the dimension of $V$ is not needed here, so this proof would work also for points in $V$, and we would get the same density as in Lemma 2.3 of \cite{poonensubscheme}:
$$\mu_Z(\mathcal P_r)=\dfrac{\# T}{\# H^0(Y,\mathcal I_Z\cdot\mO_Y)}\prod\limits_{P\in (V)_{<r}}(1-q^{(m-l)\deg P})\prod\limits_{P\in (U-V)_{<r}}(1-q^{-(m+1)\deg P})^{-1}.$$
Nevertheless, we cannot ask for smoothness in points on $V$ for the polynomials in $\mathcal P_r$ or $\mathcal P$, since the assumption on the dimension in our case, i.e. $2l\leq m+k$, does not imply convergence of the zeta function.
\end{rem}

\subsection{Points of medium degree}

\begin{lemma}[Singularities of medium degree]
\label{medium}
For $\mathcal Q_r^{\medium}$ defined as in \cref{rem:notation}, we have $\lim\limits_{r\rightarrow \infty}\mu_Z(\mathcal Q_r^{\medium})=0$.
\end{lemma}
\begin{proof}
This is parallel to the proof of Lemma 3.2. of \cite{poonensubscheme}, just ignore points in $V$. Again, the condition on the dimension is not used to prove this, but it does not work for points in $V$ as the resulting corresponding series does not converge for $2l>m$.
\end{proof}

\subsection{Points of high degree}

\begin{lemma}[Singularities of high degree off V]
\label{highoffv}
For $\mathcal Q^{\high}_{U-V}$ defined as in \cref{rem:notation}, 
$$\overline{\mu}_Z(\mathcal Q^{\high}_{U-V})=0.$$
\end{lemma}

\begin{proof}
This is Lemma 4.2. of \cite{poonensubscheme}. The assumption $m>2l$ is not used in the proof. 
\end{proof}

\begin{lemma}[Singularities of high degree on V]
\label{highonv}
For $\mathcal Q_V$ defined as in \cref{rem:notation}, $$\overline{\mu}_Z(\mathcal Q_V)=0.$$
\end{lemma}
\begin{proof}
The proof of this claim uses the induction argument used by Poonen in the proof of Lemma 2.6 in \cite{poonen}, and up to the definition of the polynomials $g_i$, it is similar to the one of Lemma 4.3 of \cite{poonensubscheme}. 

We may assume $U$ is contained in $\A^n=\left\{x_0\neq 0\right\}\subseteq \mP^n$. Dehomogenize by setting $x_0=1$, and identify $S_d$ with the space of polynomials $S_d'\subseteq \F_q\left[x_1,\ldots,x_n\right]=A$ of degree $\leq d$ and $I_d$ with a subspace $I_d'\subseteq S_d'$.

Let $P$ be a closed point of $U$. Choose a system of local parameters $t_1,\ldots,t_n\in A$ at $P$ on $\A^n$ such that $t_{m+1}=\ldots=t_n=0$ defines $U$ locally at $P$, and $t_1=\ldots=t_{m-l}=t_{m+1}=\ldots=t_n=0$ defines $V$ locally at $P$.
We may assume in addition that $t_1,\ldots,t_{m-l}$ vanish on $Z$ (cf. \cite{poonensubscheme}, Lemma 4.3).
By definition, $dt_1,\ldots,dt_n$ are a $\mO_{\A^n,P}$-basis for the stalk $\Omega^1_{\A^n|\F_q, P}$. Let $\partial_1,\ldots,\partial_n$ be the dual basis of the stalk of the tangent sheaf $\mathcal T_{\A^n|\F_q}$ at $P$. 
Choose $s\in A$ satisfying $s(P)\neq 0$ to clear denominators such that $D_i=s\partial_i$ defines a global derivation $A\rightarrow A$ for all $i$. Then there exists a neighbourhood $N_P$ of $P$ in $\A^n$ such that $N_P\cap \left\{t_{m+1}=\ldots=t_n=0\right\}=N_P\cap U$, $\Omega^1_{N_P|\F_q}=\oplus_{i=1}^n\mO_{N_P}dt_i$ and $s\in \mO(N_P)^*$. We can cover $U$ with finitely many $N_P$, so we may assume that $U$ is contained in $N_P$ for some $P$. For $f\in I_d'\cong I_d$, the hypersurface section $H_f\cap U$ fails to be smooth of dimension $m-1$ at some point $P\in V$ if and only if $(D_1 f)(P)=\ldots=(D_m f)(P)=0$. 

Let $\tau=\max_i (\deg t_i)$, $\gamma=\left\lfloor (d-\tau)/p\right\rfloor$, and $\nu=\left\lfloor d/p\right\rfloor$. If we choose $f_0\in I'_d$, $g_1\in S_{\gamma}',\ldots,g_{l-k}\in S_{\gamma}'$ uniformly at random, then
$$f=f_0+g_1^pt_1+\ldots+g_{l-k}^pt_{l-k}$$
is a random element of $I_d'$, because of $f_0$ and since by assumption $l-k\leq m-l$. Note that we cannot define more polynomials $g_{l-k+1},\ldots, g_m$ as in \cite{poonensubscheme}, since our condition on the dimension would not yield a polynomial in $I_d'$.

For $i=0,\ldots,l-k$ we define
$$W_i=V\cap \left\{D_1 f=\ldots=D_i f=0 \right\}.$$
By definition of the $D_i$, this subscheme depends only on $f_0, g_1,\ldots,g_i$.

As in Lemma 2.6 of \cite{poonen}, one can show that for $0\leq i\leq l-k-1$, conditioned on a choice of $f_0,g_1,\ldots,g_i$ for which $\dim (W_i)\leq l-i$, the probability that $\dim (W_{i+1})\leq l-i-1$ is $1-o(1)$ as $d\rightarrow \infty$.

It follows that for $i=0,\ldots,l-k$ we have $\Prob(\dim W_i\leq l-i)=1-o(1)$ as $d\rightarrow \infty$ and thus $W_{l-k}$ is at most of dimension $k$ with probability $1-o(1)$, which is what we claimed, since $W_{l-k}$ contains the points where $H_f\cap U$ is not smooth. 
\end{proof}

\begin{proof}[Proof of \cref{bertinimove}]
We have the inclusions 
$$\mathcal P\subseteq \mathcal P_r\subseteq \mathcal P\cup \mathcal Q_r^{\medium}\cup\mathcal Q_{U-V}^{\high}\cap \mathcal Q_V:$$
The first inclusion is clear. For the second, let $f\in \mathcal P_r$. If $f$ is not in $\mathcal P$, then by definition, either $H_f\cap U$ is not smooth at some point $P\in U-V$, or $\dim(H_f\cap U)_{\sing}\geq k+1$. In the first case, this point $P$ must be of some degree $\geq r$, since $f\in \mathcal P_r$, and $f\in\mathcal Q_r^{\medium}\cup\mathcal Q_{U-V}^{\high}$. For the second case, if $H_f\cap U$ is smooth at all points in $U-V$, then the singular locus of the hypersurface section is completely contained in $V$ and $f\in \mathcal Q_V$. If $H_f\cap U$ is not smooth at some point in $U-V$, then we are in the first case again.

By \cref{low}, \ref{medium}, \ref{highoffv} and \ref{highonv},
$$\mu_Z(\mathcal P)=\lim\limits_{r\rightarrow \infty}\mu_Z(\mathcal P_r)=\dfrac{\# T}{\# H^0(Y,\mathcal I_Z\cdot\mO_Y)}\zeta_{U-V}(m+1)^{-1}.$$
\end{proof}

\section{Results for V not smooth}

In the condition on the dimension we will need for the analogue of \cref{bertinimove} in the case where $V$ is not smooth, the embedding dimension of a scheme $X$ at a point $P$ will occur naturally. It is defined as $e(P)=\dim_{\kappa(P)}(\Omega^1_{X|\F_q}(P))$. Let
$$X_e=X(\Omega^1_{X|\F_q}, e)$$
be the subscheme such that a scheme morphism $f:T\rightarrow X$ factors through $X_e$ if and only if $f^*\Omega^1_{X|\F_q}$ is locally free of rank $e$. Then $X_e$ is the locally closed subscheme of $X$ where the embedding dimension of $X$ is $e$. 

The condition on the dimension for $V$ non smooth in the Bertini smoothness theorem (cf. \cite{wutz}, \cite{gunther}) for hypersurface sections containing a closed subscheme is $\max\{e + \dim V_e\}<m$ instead of $2\dim < \dim X$ for the case $V$ smooth (\cite{poonensubscheme}). Hence, one would expect an analogue of \cref{bertinimove} in the case $V$ non smooth to hold for $\max\{e + \dim V_e\}\leq m+k$. But using our methods, we cannot prove this exact analogue of \cref{bertinimove} where $V$ is not smooth, since we can only bound the dimension of the bad points of high degree in each $V_e$ and not in $V$. The only case that still works is $k=0$: 

\begin{thm}
\label{bertinivnotsmooth}
Let $X$ be a quasi-projective subscheme of $\mP^n$ and $Y$ a finite subscheme of $\mP^n$ such that $U=X-(X\cap Y)$ is smooth of dimension $m\geq 0$ over $\F_q$. Let $Z$ be a closed subscheme of $\mP^n$ such that $Z\cap Y=\emptyset$. Let $T\subseteq H^0(Y,\mathcal I_Z\cdot \mO_Y)$. 
Assume that for $V=Z\cap U$ we have $\max\{e + \dim V_e\}\leq m$. Define
\begin{align*}
\mathcal P=\{f\in I_{\homog} :\; &f\big|_Y\in T, \; (H_f\cap U)_{\sing}\subseteq Z\\ 
&\text{ and such that } \dim(H_f\cap U)_{\sing} \leq 0\}.
\end{align*}
Then $\mu_Z(\mathcal P)=\dfrac{\#T}{\#H^0(Y,\mO_Y)}\zeta_{U-V}(m+1)^{-1}$.
\end{thm}

The proof again uses the closed point sieve. We define $\mathcal P_r$, $\mathcal Q_r^{\medium}$ and $\mathcal Q_{U-V}^{\high}$ as in in the previous section. The proofs for \ref{low}, \ref{medium} and \ref{highoffv} do not need the conditions on the smoothness or dimension of $V$, and thus we only have to show the following lemma on singularities of high degree:

\begin{lemma}[Singularities of high degree for $V$ not smooth]
\label{highvnotsmooth}
Define
$$\mathcal Q=\left\{f\in I_{\homog}:\; \dim (H_f\cap U)_{\sing}\geq 1\right\}.$$
Then $\mu_Z(\mathcal Q)=0$.
\end{lemma}

\begin{proof}
We may assume $U$ is contained in $\A^n=\left\{x_0\neq 0\right\}\subseteq \mP^n$.
As there are only finitely many $V_e$, it is enough to bound the probability that a polynomial gives a singular locus of dimension $\geq 1$ if we intersect this singular locus with $V_e$.

Let $P$ be a closed point of $V_e$. Since $U$ is smooth, we can choose a system of local parameters $t_1,\ldots, t_n\in A$ on $\mathbb A^n$ such that $t_{m+1}=\ldots=t_n=0$ defines $U$ locally at $P$. Then $dt_1,\ldots,dt_n$ are a basis for the stalk of $\Omega^1_{\A^n|\F_q}$ at $P$ and
$dt_1,\ldots,dt_m$ are a basis for the stalk of $\Omega^1_{U|\F_q}$ at $P$. Using the exact sequence (\cite{hartshorne}, Section II.8)
$\mathcal I_V/\mathcal I_V^2\rightarrow \Omega_X^1\otimes \mathcal O_V\stackrel{\phi}{\rightarrow} \Omega_V^1\rightarrow 0,$
we show that $dt_1,\ldots,dt_{m-e}$ form a basis of the kernel of $\phi$ at $P$ and $dt_{m-e+1},\ldots,dt_{m}$ a basis of $\Omega^1_{V|\F_q,P}\otimes \mO_{V_e,P}$. In particular, $t_1,\ldots, t_{m-e}$ all vanish on V, since $\Omega^1_{V|\F_q}\otimes \kappa(P)\cong \mfm_{V,P}/\mfm_{V,P}^2$. Again, we may also assume in addition that $t_1,\ldots,t_{m-e}$ vanish on $Z$.

As in \cref{highonv}, we define $\partial_i$, $s$ and $D_i$ such that for $f\in I_d'\cong I_d$, the hypersurface section $H_f\cap U$ fails to be smooth of dimension $m-1$ at some point $P\in V_e$ if and only if $(D_1 f)(P)=\ldots=(D_m f)(P)=0$. 

Let $\tau =\max\limits_{1\leq i\leq {l_e}-1} (\deg t_i)$ and $\gamma=\lfloor(d-\tau)/p\rfloor$ where $l_e=\dim V_e$. We select $f_0\in I'_d$ and $g_1\in S'_{\gamma},\ldots,g_{l_e-1}\in S'_{\gamma}$ uniformly and independently at random. Then the distribution of 
$$f=f_0+g_1^pt_1 + \ldots +g_{l_e-1}^p t_{l_e}$$
is uniform over $I'_d$, since by assumption, $l_e\leq m-e$.

Now everything works as in the proof of \cref{highonv}: for $i=0,\ldots,l_e$ we define
$$W_i=V_e\cap \left\{D_1 f=\ldots=D_i f=0 \right\},$$
depending only on $f_0, g_1,\ldots,g_i$. Using the induction as in Lemma 2.6 of \cite{poonen}, we show that $W_{l_e}$ is at most of dimension 0 with probability $1-o(1)$, which is what we claimed. 
\end{proof}

\begin{proof}[Proof of \cref{bertinivnotsmooth}]
As in \cref{sec:vsmooth}, we have the inclusions
$$\mathcal P\subseteq \mathcal P_r\subseteq \mathcal P\cup \mathcal Q_r^{\medium}\cup\mathcal Q_{U-V}^{\high}\cap \mathcal Q.$$
By \ref{medium}, \ref{highoffv} and \ref{highvnotsmooth}, the error in our approximation is negligible, and \cref{low} yields \cref{bertinivnotsmooth}.
\end{proof}


\begin{thebibliography}{breitestes label}

\bibitem [Blo71] {bloch} S. Bloch, 1971, Ph.D. thesis, Columbia University. 

\bibitem[Gun15] {gunther} J. Gunther, \textit{Random hypersurfaces and embedding curves in surfaces over finite fields}, arXiv:1510.04733v1.

\bibitem [Har93] {hartshorne} R. Hartshorne, \textit{Algebraic Geometry}, Graduate Texts in Mathematics \textbf{52}, Springer, New York, 1993.


\bibitem [Poo04] {poonen} B. Poonen, \textit{Bertini theorems over finite fields}, Annals of Mathematics \textbf{160} (2004), 1099-1127.

\bibitem [Poo08] {poonensubscheme} B. Poonen, \textit{Smooth hypersurface sections containing a given subscheme over a finite field}, Math. Research Letters \textbf{15}, no. 2, 265-271.

\bibitem [Wut16] {wutz} F. Wutz, \textit{Bertini theorems for smooth hypersurface sections containing a subscheme over finite fields}, arXiv:1611.09092.

\end{thebibliography}
\end{document}